\documentclass{article}%
\usepackage{amsfonts}
\usepackage{amsmath}
\usepackage{amssymb}
\usepackage{graphicx}

\usepackage{mathrsfs} 

\usepackage{color}

\providecommand{\U}[1]{\protect\rule{.1in}{.1in}}
\numberwithin{equation}{section}
\newtheorem{theorem}{Theorem}[section]

\newtheorem{corollary}[theorem]{Corollary}

\newtheorem{definition}[theorem]{Definition}

\newtheorem{lemma}[theorem]{Lemma}

\newtheorem{proposition}[theorem]{Proposition}
\newtheorem{remark}[theorem]{Remark}

\newenvironment{proof}[1][Proof]{\noindent\textbf{#1.} }{\ \rule{0.5em}{0.5em}}

\newcommand{\spann}{\mathop{\rm span}\nolimits}

\def\N{\mathbb{N}}

\begin{document}

\title{Remarks on $j-$eigenfunctions of operators}
\author{D. E. Edmunds
\and J. Lang}
\date{}
\maketitle

\begin{abstract}
	The paper is largely concerned with the possibility of obtaining a
	series representation for a compact linear map $T$ acting between Banach
	spaces. It is known that, using the notions of $j-$eigenfunctions and $j-$%
	eigenvalues, such a representation is possible under certain conditions on $T
	$. Particular cases discussed  include those in which $T$ can be factorised through 
	a Hilbert space, or has certain $s$-numbers that are fast-decaying. The notion of $p-$compactness proves to be useful in this context; we
	give examples of maps that possess this property. 
\end{abstract}

\section{\bigskip Introduction}
In the last few years advances have been made towards the goal of being able
to represent in series form the action of a compact linear map acting between
Banach spaces. To give a brief indication of what has been acheived, let $X,Y$
be Banach spaces both of which are uniformly convex and uniformly smooth, and
let $T$ be a compact linear map from $X$ to $Y.$ Then there is a decreasing
sequence $\left\{  X_{n} \right\}  $ of linear subspaces of $X$, each with finite
codimension and with trivial intersection if the kernel\ of $T$ is trivial; there
is an associated sequence $\left\{  x_{n}\right\}  $ in the unit sphere of
$X$ such that the norm $\lambda_{n}$ of the restriction $T_{n}$ of $T$ to
$X_{n}$ is attained at $x_{n}.$ The $x_{n}$ and $\lambda_{n}$ correspond to an
\textquotedblleft eigenvector" and an "eigenvalue" respectively of a nonlinear
operator equation involving a duality map; since the $x_{n}$ have a
semi-orthogonality property in the sense of James, they are referred to as
$j-$eigenvectors, the $\lambda_{n}$ being called $j-$eigenvalues. In certain
circumstances, and in particular if either $X$ or $Y$ is a Hilbert space, it
turns out (see \cite{EE} and \cite{EL}) that if $\ker T=\{0\}$ and $T$ has range dense in $Y$, then
\[
\text{ }Tx=\sum\nolimits_{n}\lambda_{n}\xi_{n}(x)y_{n}%
\]
for all $x\in X;$ here $y_{n}=Tx_{n}/\left\Vert Tx_{n}\right\Vert $ and the
$\xi_{n}(x)$ are recursively calculable. In addition if, for
example, $Y$ is a Hilbert space, then upper and lower bounds for the
approximation numbers of $T$ can be obtained in terms of the $\lambda_{n}.$
All this is reminiscent of the position when both $X$ and $Y$ are Hilbert
spaces. Indeed in that case, the $\lambda_{n}$ are classical eigenvalues of
$\left\vert T\right\vert =\left(  T^{\ast}T\right)  ^{1/2}$ and $\left\{
x_{n}\right\}  $ is orthogonal; every eigenvalue of $\left\vert T\right\vert $
occurs in $\left\{  \lambda_{n}\right\}  $ and has finite algebraic multiplicity. When
neither $X$ nor $Y$ is a Hilbert space, series representations of this type
have been given only under rather stringent conditions.  For example, it is sufficient to require that the
Gelfand numbers of $T$ decay fast enough; the  $y_{n}$ then form
basis of $Y$.  The book \cite{EE}
provides further details and background material. The usefulness of this approach is also illustrated
here by consideration of a $\left( p,p^{\prime }\right) -$ bi-Laplacian
boundary-value problem on a bounded interval $\left( 0,b\right):$ what
emerges is an exact description of the first eigenvalue of this problem.

For the convenience of the reader, in the early part of the paper we provide
a brief but precise summary of the ideas and results described roughly
above, including inequalities relating the $j-$eigenvalues to the Gelfand
numbers of the map. The Gelfand numbers are significant, for not only do
they provide some kind of measure of compactness of a map, but the
inequalities just mentioned help to show that when a compact map $T$ acts
from $X$ to $X$ \ and has classical eigenvalues $\widehat{\lambda }%
_{n}\left( T\right)$ (repeated according to algebraic multiplicity and ordered by
decreasing modulus) then under mild additional conditions, there is a
link between these different types of eigenvalues in the sense that 
\[
\left\vert \widehat{\lambda }_{n}\left( T\right) \right\vert
=\lim_{k\rightarrow \infty }\lambda _{n}\left( T^{k}\right) ^{1/k}\text{ \ }%
\left( n\in \mathbb{N}\right) .
\]%
Further progress is made when the compact map $T:X\rightarrow Y$ is
Hilbertian, by which we mean that it can be factorised through a Hilbert
space: \ that is, there are a Hilbert space $H$ and bounded linear maps $%
A:X\rightarrow H,$ $B:H\rightarrow Y$ such that $T=B\circ A.$ Sufficient
conditions for this to hold are given:\ for example, every nuclear map is
Hilbertian.

To obtain further illustrations we introduce the notion of $p-$%
compactness of an operator given in [19], for all $p\in \lbrack 1,\infty ].$
This enables a fine tuning to be made of the property of compactness: thanks
to a result of Grothendieck, ordinary compactness corresponds to $\infty -$%
compactness; \ and if $T$ is $p-$compact for some $p$ it is $q-$compact for
all $q\geq p.$ Every compact Hilbertian map is $2-$compact.
To
remedy the dearth of simple, concrete examples illustrating the notion of $p-
$compactness we show that the standard Hardy integral operator $T$, acting
between Lebesgue spaces on $I=\left[ 0,1\right] ,$ is $p-$compact for some $%
p<\infty .$ In particular, we show that when viewed as a map from \ $L_{2}(I)
$ to itself, $T$ is $p-$compact for all $p\in (1,\infty ]$ but is not
nuclear. Another example given involves a simple Sobolev embedding. 

At the abstract level what eventually emerges  is that the action of a compact Hilbertian map can be represented in form of series which is constructed by using $j$-eigenfunctions. When the components of a Hilbertian map are compact then the map can be described by means of a double series. It also
turns out that under appropriate conditions, a Banach space  that is the
target of a compact Hilbertian map has a basis which is related to $j$-eigenfunctions.

\bigskip

\section{Preliminaries and introduction of j-eigenvalues}
In this section we recall the definition and basic properties  of $j-$eigenfunctions and $j-$eigenvalues for compact linear operators  acting between Banach spaces and their connection with s-numbers.

Throughout this section we shall suppose that $X$ and $Y$ are real, reflexive,
infinite-dimensional Banach spaces, with norms $\left\Vert \cdot\right\Vert
_{X}$,$\left\Vert \cdot\right\Vert _{Y}$ and duals $X^{\ast},Y^{\ast},$ that
are uniformly smooth and uniformly convex; $B_{X}$ stands for the closed unit
ball in $X$ and has boundary $S_{X};$ and we shall write $B(X,Y)$ for the
space of all bounded linear maps from $X$ to $Y,$ abbreviating this to $B(X)$
when $X=Y.$ Similarly, $K(X,Y)$ and $K(X)$ will stand for the spaces of
compact linear maps. The value of $x^{\ast}\in X^{\ast}$ at $x\in X$ is
denoted by $\left\langle x,x^{\ast}\right\rangle _{X}.$ Given any closed
linear subspaces $M,N$ of $X,X^{\ast}$ respectively, their polar sets are
\[
M^{0}=\left\{  x^{\ast}\in X^{\ast}:\left\langle x,x^{\ast}\right\rangle
_{X}=0\text{ for all }x\in M\right\}
\]
and%
\[
^{0}N=\left\{  x\in X:\left\langle x,x^{\ast}\right\rangle _{X}=0\text{ for
all }x^{\ast}\in N\right\}  .
\]
Note that the strict convexity of $X$ and $X^{\ast}$ is inherited by
$M,X\backslash M,\left(  X\backslash M\right)  ^{\ast}$ and $M^{0}.$

A reflexive Banach space $Z$ has strictly convex dual if and only if $\left\Vert
\cdot\right\Vert _{Z}$ is G\^{a}teaux-differentiable on $Z\backslash\{0\}.$
The G\^{a}teaux derivative $\widetilde{J}_{X}(x):=\operatorname{grad}%
\left\Vert x\right\Vert _{X}$ of $\left\Vert x\right\Vert _{X}$ at $x\in
X\backslash\{0\}$ is the unique element of $X^{\ast}$ such that%
\[
\left\Vert \widetilde{J}_{X}(x)\right\Vert _{X^{\ast}}=1\text{ and
}\left\langle x,\widetilde{J}_{X}(x)\right\rangle _{X}=\left\Vert x\right\Vert
_{X}.
\]
A gauge function is a map $\mu:[0,\infty)\rightarrow\lbrack0,\infty)$ that is
continuous, strictly increasing and such that $\mu(0)=0$ and $\lim
_{t\rightarrow\infty}\mu(t)=\infty;$ the map $J_{X}:X\rightarrow X^{\ast}$
defined by%
\[
J_{X}(x)=\mu\left(  \left\Vert x\right\Vert _{X}\right)  \widetilde{J}%
_{X}(x)\text{ \ }\left(  x\in X\backslash\{0\}\right)  ,J_{X}(0)=0,
\]
is called a duality map on $X$ with gauge function $\mu.$ For all $x\in X,$%
\[
\left\langle x,J_{X}(x)\right\rangle _{X}=\left\Vert J_{X}(x)\right\Vert
_{X^{\ast}}\left\Vert x\right\Vert _{X},\text{ }\left\Vert J_{X}(x)\right\Vert
_{X^{\ast}}=\mu\left(  \left\Vert x\right\Vert _{X}\right)  .
\]
A semi-inner product is defined on $X$ by%
\[
\left(  x,h\right)  _{X}=\left\Vert x\right\Vert _{X}\left\langle
h,\widetilde{J}_{X}(x)\right\rangle _{X}\text{ \ }\left(  x\neq0\right)
,\text{ \ }\left(  0,h\right)  _{X}=0.
\]
If $X$ is a Hilbert space, $\left(  \cdot,\cdot\right)  _{X}$ is an inner
product. Proofs of these assertions and further details of Banach space
geometry may be found in \cite{EE} and \cite{LT}. From now on we suppose that
$X$ and $Y$ are equipped with duality maps corresponding to gauge functions
$\mu_{X},\mu_{Y}$ respectively, normalised so that $\mu_{X}(1)=\mu_{Y}(1)=1.$

Next we remind the reader of the notion of orthogonality introduced by James
\cite{Jam} and mention its connection with the semi-inner product above. We
say that an element $x\in X$ is $j-$orthogonal (or orthogonal in the sense of
James) to $y\in X$, and write $x\perp^{j}y,$ if%
\[
\left\Vert x\right\Vert _{X}\leq\left\Vert x+ty\right\Vert _{X}\text{ for all
}t\in\mathbb{R}.
\]
If $x$ is $j-$orthogonal to every element of a subset $W$ of $X,$ it is said
to be $j-$orthogonal to $W,$ written $x\perp^{j}W.$ A subset $W_{1}$ of $X$ is
$j-$orthogonal to $W_{2}\subset X$ (written $W_{1}\perp^{j}W_{2})$ if
$x\perp^{j}y$ for all $x\in W_{1}$ and all $y\in W_{2}.$ In general,
$j-$orthogonality is not symmetric, that is, $x\perp^{j}y$ need not imply
$y\perp^{j}x;$ indeed, symmetry would imply that $X$ is a Hilbert space, if
$\dim X\geq3$. The connection between $j-$orthogonality and the semi-inner
product is given by the following result of \cite{Jam}: if $x,h\in X,$ then
$x\perp^{j}h$\ if and only if $\left(  x,h\right)  _{X}=0.$

An important decomposition of $X$ in terms of James orthogonality was given by
Alber \cite{Alb}, who introduced the following terminology:\ given closed
subsets $M_{1},M_{2}$ of $X,$ the space $X$ is said to be the James orthogonal
direct sum of $M_{1}$ and $M_{2},$ written $X=M_{1}\uplus M_{2},$ if

1) for each $x\in X$ there is a unique decomposition $x=m_{1}+m_{2},$ where
$m_{1}\in M_{1},$ $m_{2}\in M_{2};$

2) $M_{2}\perp^{j}M_{1};$

3) $M_{1}\cap M_{2}=\emptyset.$

Alber   \cite{Alb}  established the following:

\begin{theorem}
\label{Theorem 2.1} Let $M$ be a closed subset of a uniformly convex and
uniformly smooth space $X;$ let $J_{X}$ be the duality map with gauge function
$\mu$ given by $\mu(t)=t$ for all $t\geq0.$ Then
\[
X=M\uplus J_{X}^{-1}M^{0}\text{ and }X^{\ast}=M^{0}\uplus J_{X}M.
\]

\end{theorem}

Next, we introduce $j-$eigenfunctions and $j-$eigenvalues, and summarise some
of the results obtained in the last few years concerning the representation of
compact linear operators acting between Banach spaces; for a more complete account
of this we refer to \cite{EE}. Let $T\in K(X,Y)$ and set
\[
S_{T}(x)=\frac{\left\Vert Tx\right\Vert _{Y}}{\left\Vert x\right\Vert _{X}%
}\text{ \ }\left(  x\in X\backslash\{0\}\right)  .
\]
Then for all $s,x\in X,$
\begin{equation}
\left\langle s,\operatorname{grad}S_{T}(x)\right\rangle =\frac{d}{dt}\left(
\frac{\left\Vert Tx+tTs\right\Vert _{Y}}{\left\Vert x+ts\right\Vert _{X}%
}\right)  \Bigm|_{t=0}. \label{Eq 2.1}%
\end{equation}
Suppose that $x\in S_{X}.$ Then
\begin{equation}
\operatorname{grad}S_{T}(x)=0 \label{Eq 2.2}%
\end{equation}
if and only if
\begin{equation}
\left\langle Ts,\widetilde{J}_{Y}Tx\right\rangle _{Y}=\left\Vert Tx\right\Vert
_{Y}\left\langle s,\widetilde{J}_{X}x\right\rangle _{X} \label{Eq 2.3}%
\end{equation}
for all $s\in X.$ Equation (\ref{Eq 2.3}) can also be expressed as%
\begin{equation}
T^{\ast}\widetilde{J}_{Y}Tx=\lambda\widetilde{J}_{X}x,\text{ with }%
\lambda=\left\Vert Tx\right\Vert _{Y}, \label{Eq 2.4}%
\end{equation}
or as
\begin{equation}
T^{\ast}J_{Y}Tx=\nu J_{X}x,\text{ with }\nu=\lambda\mu_{Y}\left(  \left\Vert
Tx\right\Vert _{Y}\right)  . \label{Eq 2.5}%
\end{equation}

Since $T$ is compact, there exists $x_1 \in S_X$ such that $\left\Vert T\right\Vert =\left\Vert
Tx_{1}\right\Vert _{Y}$ and
\[
T^{\ast}\widetilde{J}_{Y}Tx_{1}=\lambda_{1}\widetilde{J}_{X}x_{1},\text{
}\lambda_{1}=\left\Vert T\right\Vert ,
\]
or, equivalently,%
\[
T^{\ast}J_{Y}Tx_{1}=\nu_{1}J_{X}x_{1},\text{ with }\nu_{1}=\left\Vert
T\right\Vert \mu_{Y}\left(  \left\Vert T\right\Vert \right)  .
\]

We can see that the dual version of the above observation with
\[
S_{T*}(y^*)=\frac{\left\Vert T^*y^*\right\Vert _{X^*}}{\left\Vert y^*\right\Vert _{Y^*}%
}\text{ \ }\left(  y^*\in Y^*\backslash\{0\}\right) , 
\]
gives us that there exists an extremal element $y_1^* \in Y^*$, with $\|y^*_1\|_{Y^*}=1$, such that $\|T^*\|=\|T^*y_1^*\|_{X^*}$
and
\begin{equation} \label{dual equation}
T{J}^{-1}_{X}T^*y^*_{1}=\nu^*_{1} {J^{-1}_{Y}}y^*_{1},\text{
}\nu^*_{1}=\left\Vert T^*\right\Vert \mu^{-1}_X(\|T^*\|).
\end{equation}

And it is quite easy to see that 
\[y^*_1= J_Y T x_1 / \mu_Y (\|T\|),
\]
from which it follows that equations (\ref{dual equation}) and (\ref{Eq 2.5}) are equivalent.

Set $X_{1}=X,M_{1}=$ sp $\left\{  J_{X}x_{1}\right\}  $ (where sp denotes the
linear span)$,$ $X_{2}=$ $^{0}M_{1},N_{1}=$ sp$\left\{  J_{Y}Tx_{1}\right\}
,$ $Y_{2}=$ $^{0}N_{1}$ and $\lambda_{1}=\left\Vert T\right\Vert .$ Since
$X_{2}$ and $Y_{2}$ are closed subspaces of the reflexive spaces $X$ and $Y$
respectively (each with codimension $1$), they are reflexive. As $X_{2}^{\ast
}=\left(  ^{0}M_{1}\right)  ^{\ast}$ is isometrically isomorphic to
$X_{1}^{\ast}/M_{1},$ it follows that $X_{2}^{\ast}$ is strictly convex; the
same argument applies to $Y_{2}^{\ast}.$ Because
\[
\left\langle Tx,J_{Y}Tx_{1}\right\rangle _{Y}=\nu_{1}\left\langle x,J_{X}%
x_{1}\right\rangle _{X}\text{ for all }x\in X,
\]
we see that $T$ maps $X_{2}$ to $Y_{2}.$ The restriction $T_{2}$ of $T$ to
$X_{2}$ is thus a compact linear map of $X_{2}$ to $Y_{2},$ and if it is not
the zero operator we can repeat the above argument: there exists $x_{2}\in
X_{2},$ with $\left\Vert x_{2}\right\Vert _{X}=1,$ such that
\[
\left\langle T_{2}x,J_{Y_{2}}T_{2}x_{2}\right\rangle _{Y}=\nu_{2}\left\langle
x,J_{X_{2}}x_{2}\right\rangle _{X}\text{ for all }x\in X_{2},
\]
where $\nu_{2}=\lambda_{2}\mu_{Y}\left(  \lambda_{2}\right)  ,$ $\lambda
_{2}=\left\Vert Tx_{2}\right\Vert _{Y}=\left\Vert T_{2}\right\Vert .$
In this, $J_{X_{2}}$ and $J_{Y_{2}}$ may be
replaced by $J_{X}$ and $J_{Y}$ respectively: see \cite{EE}, p.75.

Evidently $\lambda_{2}\leq\lambda_{1}$ and $\nu_{2}\leq\nu_{1}.$ In this way
we obtain elements $x_{1},...,x_{n}$ of $X,$ each with unit norm, subspaces
$M_{1},...,M_{n}$ of $X^{\ast}$ and $N_{1},...,N_{n}$ of $Y^{\ast},$ where%
\[
M_{k}=\text{ sp }\left\{  J_{X}x_{1},...J_{X}x_{k}\right\}  ,N_{k}=\text{ sp
}\left\{  J_{Y}Tx_{1},...J_{Y}Tx_{k}\right\}  \text{ }(k=1,...,n),
\]
and decreasing families $X_{1},...,X_{n}$ and $Y_{1},...,Y_{n}$ of subspaces
of $X$ and $Y$ respectively given by%
\[
X_{k}=\text{ }^{0}M_{k-1},Y_{k}=\text{ }^{0}N_{k-1}\text{ }\left(
k=2,...,n\right)  .
\]
For each $k\in\{1,...,n\},$ $T$ maps $X_{k}$ into $Y_{k},$ $x_{k}\in X_{k}$
and, setting $T_{k}:=T\upharpoonleft_{X_{k}},\lambda_{k}(T)=\lambda
_{k}=\left\Vert T_{k}\right\Vert ,\nu_{k}=\lambda_{k}\mu(\lambda_{k}),$ we
have%
\begin{equation}
\left\langle T_kx,J_{Y}Tx_{k}\right\rangle _{Y_{k}}=\nu_{k}\left\langle
x,J_{X}x_{k}\right\rangle _{X_{k}}\text{ for all }x\in X_{k}, \label{Eq 2.6}%
\end{equation}
and so
\[
T^{\ast}_kJ_{Y_{k}}T_{k}x_{k}=\nu_{k}J_{X_{k}}x_{k}.
\]
Identifying $Y^*_k$ with $Y^*/Y^0_k$ and $X^*_k$ with $X^*/X^0_k$ we can see that (\ref{Eq 2.6}) is equivalent to
\[
\left\langle T_kx,J_{Y}T_kx_{k}\right\rangle _{Y}=\nu_{k}\left\langle x,J_{X}%
x_{k}\right\rangle _{X}\text{ for all }x\in X_{k}.
\]
Since $Tx_{k}\in Y_{k}=$ $^{0}N_{k-1},$ we have
\[
\left\langle Tx_{k},J_{Y}Tx_{l}\right\rangle _{Y}=0\text{ if }l<k.
\]
The process stops with $\lambda_{n},x_{n}$ and $X_{n+1}$ if and only if the
restriction of $T$ to $X_{n+1}$ is the zero operator while $T_{n}\neq0.$ With
respect to the semi-inner product on $X$ the $x_{n}$ have the
semi-orthogonality property
\[
\left(  x_{r},x_{s}\right)  _{X}=\delta_{r,s}\text{ if }r\leq s,
\]
where $\delta_{r,s}$ is the Kronecker delta; correspondingly, the
$y_{n}:=Tx_{n}/\left\Vert Tx_{n}\right\Vert _{Y}=Tx_{n}/\lambda_{n}$ satisfy%
\[
\left(  y_{r},y_{s}\right)  _{Y}=\delta_{r,s}\text{ if }r\leq s.
\]
This means that $x_{i}\perp^{j}X_{k}$ and $y_{i}\perp^{j}Y_{k}$ whenever
$i < k.$ We shall refer to the $\lambda_{i}$ and $x_{i}$ as $j-$eigenvalues
and $j-$eigenvectors respectively of $T.$\ (Note that if greater precision is
desirable, we write $x_{i}^{T},y_{i}^{T},\lambda_{i}^{T}$ instead of
$x_{i},y_{i},\lambda_{i}$.)

Using reflexivity of $X$ and $Y$ we can introduce, for $T^*:X^* \to Y^*$, a corresponding dual problem and define
$x_{i}^*,y_{i}^*$ and $\lambda_i^*$. Note that since $\|T_k\|=\|T^*_k\|$ we have 
\begin{equation}
\lambda_i=\lambda_i^* \mbox{ and } x^*_i=J_{X_i} x_i, \, y^*_i=J_{Y_i} y_i. \label{duality-lambda} 
\end{equation}

 If $T$ has infinite rank, then $\left\{
\lambda_{n}\right\}  $ is an infinite sequence that converges monotonically to
zero, and
\[
\cap_{n=1}^{\infty}X_{n}=\ker T=\{0\}.
\]

The natural question which arises in the context of $j-$eigenfunctions is whether a compact $T$ can be described in form of a series based on $j-$eigenfunctions. We provide more  comments about this in the next sections.

One can also ask a question about the general relations between j-eigenfunctions and s-numbers which are used in the study of compact maps. The only result known to us (see \cite{EL1}) is
\begin{lemma} \label{Lemma c_n}
Let $T\in K(X,Y)$. Then for all $k\in \N$, 
\[(2^k-1)^{-1} \lambda_k(T) \le c_k(T) \le \lambda_k(T),
\]	
where $c_k(T)$ is the $k$-th Gelfand s-number of $T$
defined by %
\[
c_{k}(T)=\inf\left\{  \left\Vert TJ_{M}^{X}\right\Vert :\text{ codim
}M<k\right\}  ,
\]
and $J_{M}^{X}$ is the canonical injection from the subspace $M$ to $X.$
\end{lemma}

\section{Generalization of Hilbert-Schmidt Theorem (a)}

Let us consider  the situation when a compact operator $T$ is a map from a Banach space to a Hilbert space (or in the opposite order). With the aid of $j-$eigenvalues it is possible to obtain, in this set-up,  a generalization of the Hilbert-Schmidt decomposition theorem 
which we now recall, correcting some misprints in \cite{EL}.

Corollary 2.2.20 in \cite{EE} give us the following Theorem except the last statement
about the basis, but that follows instantly due to the orthogonality of the $h_i$ and the density of the range of $T$ in $H$. 

\begin{theorem}
\label{Theorem 2.2}Let $X$ be an infinite-dimensional Banach space that is
uniformly convex and uniformly smooth, let $H$ be a Hilbert space and suppose
that $T\in K(X,H)$ has trivial kernel and range dense in $H$. Then the
sequence $\left\{  \lambda_{i}\right\}  $ of $j-$eigenvalues of $T$ decreases
to $0,$ and the action of $T$ can be described in terms of the corresponding
sequence $\left\{  x_{i}\right\}  $ of $j-$eigenvectors of $T$ and the
sequence $\left\{  h_{i}=Tx_{i}\right\}  $ by
\begin{equation}
Tx=\sum\nolimits_{i=1}^{\infty}\lambda_{i}\xi_{i}(x)h_{i}\text{ \ }\left(
x\in X\right)  , \label{Eq 2.7}%
\end{equation}
where $\xi_{i}(x)=\lambda_{i}^{-1}\left(  Tx,h_{i}\right)  _{H}.$ Also $\{h_i\}$ forms a basis of $H$.
\end{theorem}

\begin{theorem}
\label{Theorem 2.3}Let $Y$ be an infinite-dimensional Banach space that is
uniformly convex and uniformly smooth, let $H$ be a Hilbert space and suppose
that $T\in K(H,Y)$ has trivial kernel and range dense in $Y$. Then the sequence $\left\{  \lambda
_{i}\right\}  $ of $j-$eigenvalues of $T$ decreases to $0,$ and the action of
$T$ can be described in terms of the corresponding sequence $\left\{
h_{i}\right\}  $ of $j-$eigenvectors\ of $T$ and the sequence $\left\{
y_{i}=Th_{i}/\left\Vert Th_{i}\right\Vert _{Y}\right\}  $ by
\begin{equation} 
Th=\sum\nolimits_{i=1}^{\infty}\lambda_{i}\left(  h,h_{i}\right)  _{H}%
y_{i}\text{ \ }\left(  h\in H\right)  .\label{Eq 2.8}%
\end{equation}
Here $\{h_i\}$ is a basis of $H$ and $\{y_i\}$ is a basis of $Y$. 
\end{theorem}
\begin{proof}
Combining the obvious estimate 
\begin{equation}
\|P_n\|_{H_{k-1} \to H_k}\le 1 \label{Pn}
\end{equation}  with Lemma 2.2.10 in \cite{EE} guarantees that $\|S_n\| \le 1$ and then from  Remark 2.2.33 in \cite{EE}  one can see that $\{h_i \}$ is a basis of $H$ and from this we obtain (\ref{Eq 2.8}). Now using (\ref{Pn}) and the relation between $H_n$ and $Y_n$ we can obtain that
\[\sup\{\|y\|_Y: y\in Q_n Q_{n-1} ... Q_2(T(B_H))\} \le \|T\|.
\]
 Use of Lemma 2.2.21 in \cite{EE} shows  that $\|y-R_ny\| \le \|T\|$ for all $y\in T(B_H)$. 
 Following arguments from the proof of Theorem 2.2.39 in \cite{EE} we see that $\{y_i\}$ is a basis of $Y$.
\end{proof}

One obvious question which arises is whether is it possible, in the case $T\in K(X,H)$,  to replace $\xi_i$ in (\ref{Eq 2.7}) by a linear function, the construction of which does not  directly involve $T$. With the help of the above theorem we answer, positively, the above question in the next statement which can be seen as an improvement of Theorem \ref{Theorem 2.2}.

\begin{theorem}
	\label{Theorem 2.4}Let $X,H$ and $T\in K(X,H)$ be as in Theorem
	\ref{Theorem 2.2}. Denote by $\{x_i^T\}$ and $\{\lambda_i^T\}$ the $j$-eigenfunctions and $j$-eigenvalues respectively of $T$; and $h_i^T=T(x_i^T)/ \lambda_i^T \in H$. Moreover, $\{h_i^{T^*}\}$ and  $\{\lambda_i^{T^*}\}$ are the $j$-eigenfunctions and $j$-eigenvalues respectively of $T^*$, and 
 $x_i^{T^*}=T^*(h_i^{T^*})/ \lambda_i^{T^*}$. Then $h_i^T=h_i^{T^*}$ and for each $x\in X,$ we have\
	\begin{align}
	Tx & =\sum\nolimits_{i=1}^{\infty}\lambda_{i}^{T^{\ast}}\left\langle
	x,x_{i}^{T^{\ast}}\right\rangle _{X}h_{i}^{T^{\ast}} \\
	& =\sum\nolimits_{i=1}%
	^{\infty}\lambda_{i}^{T^{\ast}}\left\langle x,J_{X}z_{i}\right\rangle
	_{X}h_{i}^{T^{\ast}} \\
		& =\sum\nolimits_{i=1}^{\infty}\lambda_{i}^T 
	\left\langle x,J_{X_i}( x_{i}^{T})\right\rangle _{X_i}h_{i}^{T}
		\end{align}
	
	where $z_{i}\in X$ and $J_{X}z_{i}=x_{i}^{T^{\ast}}.$ Also $\{h_i^T\}$ is a basis of $H$.
\end{theorem}
\begin{proof}
By Theorem \ref{Theorem 2.3},
for all $h\in H$ we have, with the obvious notation,
\[
T^{\ast}h=\sum\nolimits_{i=1}^{\infty}\lambda_{i}^{T^{\ast}}\left(
h,h_{i}^{T^{\ast}}\right)  _{H}x_{i}^{T^{\ast}}\text{.}%
\]
Using the identity
\[
\left(  Tx,h\right)  _{H}=\left\langle x,T^{\ast}h\right\rangle _{X}=\left(
T^{\ast\ast}x,h\right)  _{H}%
\]	
we obtain
\[Tx  =\sum\nolimits_{i=1}^{\infty}\lambda_{i}^{T^{\ast}}\left\langle
x,x_{i}^{T^{\ast}}\right\rangle _{X}h_{i}^{T^{\ast}}. \]
By setting $z_{i}:= J^{-1}_{X} x_{i}^{T^{\ast}} \in X$ we instantly get
\[Tx=  \sum\nolimits_{i=1}%
^{\infty}\lambda_{i}^{T^{\ast}}\left\langle x,J_{X}z_{i}\right\rangle
_{X}h_{i}^{T^{\ast}}. \]

From the construction of $x_i^T,$ $x_i^{T^*},$ $h_i^T,$ $h_i^{T^*}$  and (\ref{duality-lambda}) we have $J_{X_i} (x_i^T) = x_i^{T^*}$, $h_i^T=h_i^{T^*}$ and $\lambda_1^T=\lambda_1^{T^*}$ which gives us:

\[Tx= \sum\nolimits_{i=1}^{\infty}\lambda_{i}^T 
\left\langle x,J_{X_i}( x_{i}^{T})\right\rangle _{X_i}h_{i}^{T},
\]
and note that from  Theorem \ref{Theorem 2.3} follows that $\{h_i^T\}$ is a basis of $H$.
\end{proof}

\begin{remark}
\label{Remark 2.5}
Note that if $x\perp^{j}y$ in $X,$ then $J_{X}(y)$ $\perp^{j}J_{X}(x)$ in
$X^{\ast}.$ 
\end{remark}
This remark follows from the simple observation that if $x\perp^{j}y,$ then $\left\langle y,J_{X}(x)\right\rangle
_{X}=0$ and
\[
0=\left\langle y,J_{X}(x)\right\rangle _{X}=\left\langle J_{X^{\ast}}\left(
J_{X}(y)\right)  ,J_{X}(x)\right\rangle _{X},
\]
which implies that $J_{X}(y)$ $\perp^{j}J_{X}(x)$ in $X^{\ast}.$

The behaviour of the $z_{i}$ is different from that of the $x_{i}^{T^{\ast}},$
for if $i<k$ then $x_{i}^{T^{\ast}}\perp^{j}x_{k}^{T^{\ast}}$ while
$z_{k}\perp^{j}z_{i}.$ While $x_{1}=z_{1}$ the
direct relationship between $x_{i}^{T^{\ast}}$ and $z_{i}$ 
when $i>1$ is not clear to us, except the obvious relation $z_i=J_X^{-1}x_i^{T^*}= J_X^{-1} J_{X_i}(x_i^T)$.

We conclude this section with some observations on $j-$eigenvalues.
We will establish a connection between the  j-eigenvalues and the "standard" eigenvalues which resembles results  of K\"onig for eigenvalues.

\begin{theorem}
	Let $X$ be uniformly convex and uniformly smooth and let $T:X \to X$ be compact, with trivial kernel and dense range. Let the eigenvalues of $T$, arranged by decreasing modulus, be denoted by $\hat{\lambda}_n(T)$, and let ${\lambda}_n(T)$ be j-eigenvalues  produced by the above scheme. Then for each $n \in \N$,
	\[| \hat{\lambda}_n(T)| = \lim_{k\to \infty} \lambda_n(T^k)^{1/k}.
	\]  
\end{theorem} 
\begin{proof}
	Plainly $\ker (T^k)=\{0\}$ for all $k\in \mathbb{N}$. Moreover the range of $T^k$ is dense in $X$ for each $k\in \mathbb{N}$. 
	
	To verify this, let $y\in X$. Then there is a sequence $\{x_n\}$ in $X$ such that $Tx_n \to y$, and there is a sequence $\{y_n\}$ in $X$ such that $\|Ty_n-x_n\| \le 1/n$ for each $n\in \mathbb{N}$. Thus
	\[\|T^2y_n-y\|=\|T(Ty_n-x_n)+ Tx_n-y\| \le \|T\| \|Ty_n-x_n\| +\|Tx_n-y\| \to 0
	\]
	as $n \to \infty$. Hence the range of $T^2$ is dense in $X$. The claim now follows by induction.
	
	Since $T^k$ has the same properties as $T$, we know (see Lemma \ref{Lemma c_n}) that for all $n \in \N,$
	\[ (2^n-1)^{-1} \lambda_n(T^k)\le c_n(T^k) \le \lambda_n(T^k).
	\] 
	A result of K\"onig (see \cite[Proposition 2.d.6, p. 134]{Ko2}  and \cite{Ko1}) gives
	\[|\hat{\lambda}_n(T)| = \lim_{k\to \infty} c_n(T^k)^{1/k}\]
	and so
	\[|\hat{\lambda}_n(T)| = \lim_{k\to \infty} \lambda_n(T^k)^{1/k}.\]
\end{proof}

To conclude this section we consider a map $T\in K(H,Y),$ where $H$ is a
Hilbert space and, as before, $Y$ is infinite-dimensional, uniformly convex
and uniformly smooth; suppose that $\ker T=\{0\}$ and the range of $T$ is dense in $Y$. Given any $n\in\mathbb{N},$
the $n^{th}$ approximation number of $T$ is defined to be
\[
a_{n}(T)=\inf\left\Vert T-F\right\Vert ,
\]
where the infimum is taken over all $F\in B(H,Y)$ with rank $F<n.$ It is well
known (see \cite{Pie2}, Proposition 2.4.6) that for such a map $a_{n}(T)$
coincides with the $n^{th}$ Gelfand number $c_{n}(T).$

Moreover (see \cite{Pie2}, Lemma 2.7.1), the following upper estimates for
approximation numbers hold.

\begin{lemma}
	\label{Lemma 3.7}Let $Y$ be a Banach space, let $S\in B(H,Y)$ and suppose that
	$a_{n}(S)>0$ for some $n\in\mathbb{N}.$ Then given any $\varepsilon>0$ there
	exists an orthonormal family $\left(  x_{1},...,x_{n}\right)  $\ in $H$ such
	that
	\[
	a_{k}(S)\leq(1+\varepsilon)\left\Vert Sx_{k}\right\Vert _{Y}\text{ for
	}k=1,...,n.
	\]
	
\end{lemma}

For our map $T$ an improvement of this result is possible as a simple consequence of Theorem {\ref{Theorem 2.3}}.

\begin{lemma}
	\label{Lemma 3.8}Let $\left\{  h_{i}\right\}  $ and $\left\{  \lambda
	_{i}\right\}  $ be the sequences of $j-$eigenfunctions and \newline%
	$j-$eigenvalues respectively of $T$ (see Theorem \ref{Theorem 2.3}). Then the
	orthonormal sequence $\left\{  h_{i}\right\}  $ has the property that
	\[
	a_{i}(T)\leq\left\Vert Th_{i}\right\Vert _{Y}=\lambda_{i}\text{ for all }%
	i\in\mathbb{N}.
	\]
	
\end{lemma}

\begin{proof}
	Since $a_{k}(T)=c_{k}(T)$ and, from the definition of the Gelfand numbers,
	$c_{k}(T)\leq\lambda_{k},$ the result is clear.
\end{proof}

Since $T\in K(H,Y)$ guarantees $a_{k}(T)=c_{k}(T)$ we have instantly (from Lemma \ref{Lemma c_n})

\begin{corollary}
	\label{Corollary 3.9} Let operator $T\in K(H,Y)$ and $j-$eigenvalues be as in Theorem \ref{Theorem 2.3}. Then for all $n\in\mathbb{N},$%
	\[
	\left(  2^{n}-1\right)  ^{-1}\lambda_{n}\leq c_{n}(T)=a_{n}(T)\leq\lambda
	_{n}.
	\]
	
\end{corollary}

\section{Generalization of Hilbert-Schmidt Theorem (b)}

When one consider a compact linear map $T$ acting between Banach spaces then, due to the famous result of P. Enflo, it is futile to hope to express $T$ in series form without additional conditions. Obviously uniform convexity and uniform smoothness conditions on Banach spaces are not sufficient (Enflo's counterexample is based on a subspace of a uniformly smooth and uniformly convex space). 

One natural condition, under which we can generalize the Hilbert-Schmidt theorem, is to strengthen the compactness of $T$ by imposing the decay condition on the $\lambda_i$ stated in the next theorem (see \cite{EL1} and Theorems 2.2.39 and 2.2.41 in \cite{EE}).

\begin{theorem} \label{Th H-S th with lambda}
	Let $X$ and $Y$ be infinite-dimensional Banach spaces that are
uniformly convex and uniformly smooth and suppose
that $T\in K(X,Y)$ has trivial kernel and range dense in $Y;$ let $\{x_i\}$, $\{y_i\}$ and $\{\lambda_i\}$ be corresponding $j-$eigenfunctions and $j-$eigenvalues of $T$. Suppose that 
\begin{equation}
\lambda_n \le 2^{-n+1} \mbox{ for all } n\in \N \label{lambda_n <}.
\end{equation}
Then  for all $x\in X,$
\begin{equation}
Tx=\sum\nolimits_{i=1}^{\infty}\lambda_{i}\xi_{i}(x)y_{i}  , \label{Eq Gen H-S lambda_n}%
\end{equation}
where $\xi_{i} \in X^*$ is obtained as a linear projection (for definition see \cite[(2.2)]{EL1}). Also  $\{y_i\}$ is a basis on $Y$. 
\end{theorem}

Since it is rather difficult, in general, to analyse $\lambda_i$, it is desirable to replace the decay of $\lambda_i$ by the decay of a more familiar object associated with the compactness of $T$ (like $s-$numbers or entropy numbers). This can be accomplished by combining  Lemma {\ref{Lemma c_n}} with the above Theorem: we thus obtain a series representation of
compact operators whose Gelfand numbers decay fast enough. 

\begin{lemma} \label{Lemma H-S with c_n}
Let $X,$ $Y$ and $T$ be as in Theorem \ref{Th H-S th with lambda}; suppose that $c_n(T) \le 2^{-n+1}(2^n-1)^{-1}$ for all $n\in \N$. Then for all $x\in X,$
	\begin{equation}
		Tx=\sum\nolimits_{i=1}^{\infty}\lambda_{i}\xi_{i}(x)y_{i} , \label{Eq Gen H-S c_n}%
	\end{equation}
	where $\xi_{i} \in X^*$ and is obtained as a linear projection (see \cite[(2.2)]{EL1}). Also $\{y_i\}$ is basis of $Y$. 
\end{lemma}    

We note that the condition $\lambda_n \le 2^{-n+1}$ on the $j-$eigenvalues guarantees that the map $T\in L(X,Y)$ is  nuclear and obviously compact (see Remark 2.2.42 in \cite{EE}) which demonstrates that the conditions (\ref{lambda_n <}) is quite strict. We should note that even in the case when $T$ is nuclear (as in the previous statement) the novelty of  $j-$eigenfunctions is that they provide us with a quite straightforward  algorithm for constructing the series (\ref{Eq Gen H-S c_n}) and obtaining it in terms of $j-$eigenfunctions and $j-$eigenvalues. 

An obvious question which arises is whether it is possible to weaken the upper estimate in  (\ref{lambda_n <}). 

In \cite{EL3} the improvement stated in the next theorem was achieved by taking $X=Y=L_p(I)$ and using information about the inner structure of $L_p$ spaces, i.e. information about the value of the minimal projection constant on subspaces with codimension 1, which is expressed as a constant  $\alpha_p$ in this expression:
\begin{equation}
	1+\alpha_p= \max_{m\in(0,1)} \left( m^{p/p'}+(1-m)^{p/p'} \right)^{1/p}
	\left( m^{p'/p} +(1-m)^{p'/p}  \right)^{1/p'}.
\end{equation}
Note that this constant is related to the polar projection constant for projections on $L_p(I)$ spaces and we have, obviously,  $\alpha_2=0$.  

\begin{theorem} \label{Th H-S on Lp th with lambda}
	Let $X=Y=L_p(I),$ suppose that $T\in L(X,Y)$ has trivial kernel and range dense in $Y$, and let $\{x_i\}$, $\{y_i\}$ and $\{\lambda_i\}$ be the corresponding $j-$eigenfunctions and $j-$eigenvalues of $T$. Suppose that 
	\begin{equation}
		\lambda_n \le (1+\alpha_p)^{-n+1} \mbox{ for all } n\in \N. \label{lambda_n <=}
	\end{equation}
	Then for all $x\in X,$
	\begin{equation}
		Tx=\sum\nolimits_{i=1}^{\infty}\lambda_{i}\xi_{i}(x)y_{i}  , \label{Eq Gen H-S Lp lambda_n}%
	\end{equation}
	where $\xi_{i} \in X^*$ and   $\{y_i\}$ is a basis of $Y$. 
\end{theorem}

We can see that the upper estimates for $\lambda_n$  "improve" when $p \to 2$ but in some sense this improvement is not fundamental. 

In view of the above results one may wonder whether apart from conditions about the geometry of Banach spaces or speed of decay for $\lambda_n$ there is some other more  simple condition under which one can express a map $T$ as a series. In view of Theorems \ref{Theorem 2.2}, \ref{Theorem 2.3} and \ref{Theorem 2.4} restriction to compact operators which can be factorized via some nice spaces seems to be reasonable.

Accordingly we focus on operators that can be factorised through a Hilbert space, i.e. Hilbertian operators.

\begin{definition}
	\label{Definition 3.2} Given Banach spaces $X$ and $Y$, a map $T\in B(X,Y)$ is said to be Hilbertian if there are
	a Hilbert space $H$ and maps $B\in B(H,Y)$ and $A\in B(X,H)$ such that
	$T=B\circ A.$ The space of all such maps is denoted by $\Gamma_{2}(X,Y);$
	endowed with the norm $\gamma_{2},$ where
	\[
	\gamma_{2}(T)=\inf\left\{  \left\Vert A\right\Vert \left\Vert B\right\Vert
	\right\}  ,
	\]
	where the infimum is taken over all possible such factorisations of $T,$ it is
	a Banach space.
\end{definition}
From Theorem 5.2 of \cite{LP} we have this nice observation:
\begin{corollary} \label{Hilbertian maps}
Let  $X=L_p(\mu)$ (or $X=l_p$) and $Y=L_q(\nu)$ (or $Y=l_q$) with $1\le q \le 2 \le p \le \infty.$ Then every $T\in B(X,Y)$ is Hilbertian.
\end{corollary}

More information
about Hilbertian maps is provided in \cite{Pis} and \cite{LP}.

As in section 2, $X,Y$ are assumed to be
real, infinite-dimensional Banach spaces that are uniformly convex and
uniformly smooth. Let $T\in K(X,Y)$ be Hilbertian, with trivial kernel. Then
there are a Hilbert space $H$ and maps $A\in B(X,H),$ $B\in B(H,Y)$ such that
$T=B\circ A:$ no compactness is required of $A$ or $B$.

Let us suppose that $A$ and $B$ are compact maps. Then following the procedure described in section 2 we obtain
sequences $\left\{  x_{i}^{A}\right\}  ,\left\{  h_{i}^{A}\right\}  ,\left\{
\lambda_{i}^{A}\right\}  $ \ corresponding to $A\in K(X,H),$ while to $B\in
K(H,Y)$ there correspond $\left\{  h_{i}^{B}\right\}  ,\left\{  y_{i}%
^{B}\right\}  ,\left\{  \lambda_{i}^{B}\right\}  ;$ in the same way, $\left\{
h_{i}^{A^{\ast}}\right\}  ,$ $\left\{  x_{i}^{A^{\ast}}\right\}  ,\left\{
\lambda_{i}^{A^{\ast}}\right\}  $ and $\left\{  y_{i}^{B^{\ast}}\right\}
,\left\{  h_{i}^{B^{\ast}}\right\}  ,\left\{  \lambda_{i}^{B^{\ast}}\right\}
$ correspond to $A^{\ast}\in K(H,X^{\ast})$ and $B^{\ast}\in K(Y^{\ast},H)$
respectively. \ The elements of the first two sequences in each group of three
all have unit norm in the appropriate space. As a direct consequence of  Theorems \ref{Theorem 2.2}, \ref{Theorem 2.3} and \ref{Theorem 2.4} we obtain instantly the following two statements.

\begin{corollary} \label{Corollary 4.6}
Let $T=B \circ A$ be as above and suppose $A\in K(X,H)$ and $B\in K(H,Y)$ where $A$ and $B$ have trivial kernels and ranges dense in the target spaces. Then 
 \begin{align*}
 	T(x)& =BA(x)= \sum_{j=1}^{\infty} \sum_{i=1}^{\infty} \lambda_j^B \lambda_i^{A^*} y_j^B\left\langle x, x_i^{A^*} \right\rangle_X \left\langle    h_i^{A^*}, h_j^B \right\rangle_H , \qquad \left(  x\in X\right)  .
 \end{align*}
 If $\left\{  \lambda_{i}^{A^*}\right\}  $ and $\left\{  \lambda
 _{i}^{B}\right\}  $ belong to $l_{2}$,  the  terms may be rearranged as we please.
\end{corollary}

\begin{corollary} \label{Corollary 4.7}
 Let $T=B \circ A$ be as above and suppose $A\in K(X,H)$, and $B\in B(H,Y)$. Then 
 \begin{align*}
 	T(x)  = \sum_{i=1}^{\infty} B(h_i^{A^*})  \lambda_i^{A^*} \left\langle x,x_i^{A^*} \right\rangle_X , \qquad \left(  x\in X\right) .
 \end{align*} 
\end{corollary}

\begin{corollary} \label{Corollary 4.8}
	 Let $T=B \circ A$ be as above and suppose $A\in B(X,H)$, and $B\in K(H,Y)$. Then 
	 \begin{align*}
	 	T(x)  = \sum_{j=1}^{\infty} \left\langle x, J_Xx_j^{C} \right\rangle_X \lambda_j^C \lambda_j^B  y_j^B \qquad \left(  x\in X\right) , 
	 \end{align*}
 where $x_j^C \in X$ and $\lambda_j^C \in l_{\infty}$ are defined in the proof.
\end{corollary}
\begin{proof}
	Define $C_j(x):= \left \langle A(x), h_j^B \right\rangle_H $ . We can see that $C_j \in X^*$ and then there exist $x_j^C \in  X^*$, $\|x_j^C\|_X=1$ and $\lambda_j^C \in \mathbb{R}$ for which
	$C_j(x)= \left \langle x, J_X x_j^{C} \right \rangle_X \lambda_j^C$.
	Note that from the boundedness of $A$ we have $\lambda_J^C \in l_{\infty}$. Using Theorem \ref{Theorem 2.3} we have:
	
	\begin{align*}
		T(x) & =BA(x)= \sum_{j=1}^{\infty} \left\langle A(x), h_j^B \right\rangle_H \lambda_j^B y_j^B \\
		& = \sum_{j=1}^{\infty} \left\langle x, J_Xx_j^{C} \right\rangle_X \lambda_j^C \lambda_j^B  y_j^B .
	\end{align*}
\end{proof}

From the above corollaries it follows that when $A$ and $B$ are compact we
have an expression of $T$ in the form of a double series which is convenient
to manage only when $\left\{ \lambda _{i}^{A^{\ast }}\right\} $ and  $%
\left\{ \lambda _{i}^{B}\right\} $ belong to $l_{2},$ and the series does
not involve $A$ or $B,$ \ directly or indirectly. When only one of $A$ and $B
$ is compact we can express $T$ in the form of a series which involves the
structure of the maps $A$ or $B$ either in the direct form (Corollary \ref{Corollary 4.7})
or the indirect form (Corollary \ref{Corollary 4.8}).

In the next theorem we consider the case of a compact Hilbertian operator $T=B \circ A$ without having any information about compactness of $A$ or $B$.

\begin{theorem}
	Let $X$, $Y$ be uniformly smooth and uniformly convex Banach spaces and $T\in K(X,Y) $ be  Hilbertian with $\ker T= \{ 0 \}$ and range dense in $Y$. Then 
	\begin{equation} \label{The main result}
	T(x)= \sum_{i=1}^{\infty} \lambda_i^{\perp} y_i^{\perp} \langle x, x_i'^{\perp} \rangle_X,  \qquad \left(  x\in X\right), 
\end{equation}
	where $\|y_i^{\perp}\|_Y=1$, $\|x_i'^{\perp}\|_{X^*}=1$ and $ \{ \lambda_i^{\perp} \} \in l_{\infty}$ and their definitions are in the proof.
	
	Note that for $T_n(x):= \sum_{i=1}^{n} \lambda_i^{\perp} y_i^{\perp} \langle x, x_i'^{\perp} \rangle_X,$ and $T^n(X):=T(X) - T_n(X)$ we have that $T^n:X \to Y_{n+1}$ and $T_n: X \to \spann\{y_1, ..., y_n\}$ from which follows that the series (\ref{The main result}) is independent of the particular Hilbertian decomposition of $T$. Also we have that $\{y_i^{\perp}\}$ is a basis of $Y$.
\end{theorem}
\begin{proof}
Let  $T=B \circ A$ be a compact map with $\ker T=\{ 0 \}$ and range dense in $Y,$ and where $A\in B(X,H)$ and $B\in B(H,Y)$ are only bounded. Obviously we can expect that $\ker A = \{ 0 \} = \ker B$ and that the range of $A$ is dense in $H$.

We construct for $T$,  by the method described in section 2,  sequences $\{x_i\}$, $\{y_i\}$, $\{X_i\} $ and $\{Y_i\}$. Define $H_i=A(X_i)$, $h_i=A(x_i)/\| A(x_i)\|_H$ and $h_i^\perp$ as a unit vector from $H$ for which $H_i= \spann\{h_i^\perp\} \oplus H_{i+1}$ and $h_i^\perp \perp H_{i+1}$. We can see that $h_i^\perp \perp h_{i+1}^\perp $.

Using the above conditions on the kernel and range of $T,A$ and $B$ (mainly that range of $A$ is dense in $H$ and $T$ is compact)  we have that co-dimension of $H_i$ is $i-1$ and also that $\{h_i^\perp \}$ is a basis in $H$.
Let $x \in X$, $h\in H$ and $y\in Y$ such that $T(x)=y$, $A(x)=h$ and $B(h)=y$.  Then we have:

$$T(x)=B(A(x))= B(\sum \langle A(x), h_i^\perp \rangle_H  h_i^\perp)=
\sum \langle A(x), h_i^\perp \rangle_H  B(h_i^\perp)
$$
which gives us a series representation of $T$. Note that the idea of proof that $\{y_i\}$ is a basis of $Y$ from  Theorem \ref{Theorem 2.3} can be quite simply modified to fit the context of  this theorem by using the observation  that the norms of the projections from $H_i$ to $H_{i+1}$ in direction of $h_i^\perp$ are equal to 1, which guarantees that the norms of projections from $Y_i$ to $Y_{i+1}$ in direction of $y_i^\perp$, restricted to the image of unit ball under the map $B$, are less or equal to 1. 
Using the observation that $A^*(h)=\sum A^*(h_i^\perp) \langle h, h_i^\perp \rangle_H$ we can write:
\begin{align*}
	\langle   x, A^*(h) \rangle_X & = \langle x, \sum A^*(h_i^\perp) \langle h, h_i^\perp \rangle_H \rangle_X 
	= \sum \langle h, h_i^\perp \rangle_H \langle x, A^*(h_i^\perp)  \rangle_X \\
	& =\langle  h,\sum h_i^\perp \langle x, A^*(h_i^\perp) \rangle_X \rangle_H .
\end{align*}

Then $A(x)=\sum h_i^\perp \langle x, A^*(h_i^\perp) \rangle_X$
and we have:
$$T(x)=(BA(X))= \sum B(h_i^\perp) \langle x, A^*(h_i^\perp) \rangle_X.$$

Now we set ${x'}_i^{\perp}:={A^*}(h_i^{\perp})/\| A^*(h_i^{\perp})\|_{X^*}$, $y_i^{\perp} := B(h_i^\perp)/ \|B(h_i^\perp) \|_Y$, and $\lambda_i^{\perp}:= \|B(h_i^\perp) \|_Y \| A^*(h_i^{\perp})\|_{X^*}$ and we have (\ref{The main result}). 

Set $T_n(x)=\sum_{i=1}^n B(h_i^\perp) \langle x, A^*(h_i^\perp) \rangle_X$ and $T^n(x)=T(x)-T_n(x)$. Then from the definition and properties of $\{y_i\} $ and $\{Y_i\}$ it follows that $T^n:X \to Y_{n+1}$ and $T_n: X \to \spann \{ y_1, ..., y_n\}$.
\end{proof}

\
\\
{\bf Examples} 
\\
Let us consider the integral operator $Tf(x)=\int_0^x k(x,y) f(y) dy$ as a map from $L_p(I)$ into $L_q(I)$ where $1\le q\le 2 \le p \le \infty,$ $I=(0,1)$ and $k$ is defined below. When $T$  is compact, then in view of the above results we have
\[Tf(x)=\sum_{i=1}^{\infty} \lambda_i g_i(x) \int_I f(y) |f_i(x)|^{p-2}f_i(x) dx
\]
where $\|f_i\|_q=\|g_i\|_p=1$ and  $f_i, g_i$ are related to $j-$eigenfunctions of $T$. A natural question which arises in this case concerns the possible relation between the partial sum of the above series and the approximation numbers of $T$ apart from the obvious estimate $a_n(T)\le \lambda_n$. We leave this  question for further research.

In the next example we demonstrate the consequences of (\ref{duality-lambda}) and Theorem \ref{Theorem 2.4} for a particular Hilbertian operator. Suppose $1<p<\infty$ and $b>0$ and define the Hardy-type operator $H:L_p(0,b) \to L_2(0,b)$ and the corresponding dual operator  $H^*:L_2(0,b) \to L_{p'}(0,b)$ by
\[Hf(x)=\int_0^x f(t) dt, \qquad H^*f(x)=\int_x^b f(t) dt.
\]
Then from \cite[Theorem 4.6]{EL2} we have 
\[\|H:L_p\to L_2\|= b^{1-1/p+1/2}(p'+2)^{1-1/p'-1/2}(p')^{1/2}2^{1/p'}/B(1/p',1/2)
\] with the unique extremals being all non-zero multiples of the generalised trigonometric functions $\cos_{p,2}({\pi_{p,2}x \over 2b})$ (see \cite{EGL} for background information about these functions) and 
\[ \|H^*:L_2\to L_{p'}\|= b^{1-1/2+1/p'}(2+p')^{1-1/2-1/p'}2^{1/p'}(p')^{1/2}/B(1/2,1/p')
\] with the unique extremals being all non-zero  multiples of $\cos_{2,p'}({\pi_{2,p'}(b-x) \over 2b})$.
From the above definition of a $j$-eigenfunction we see that  the extremal functions for maps $H$ and $H^*$ correspond to the first eigenfunctions of the $(p,2)$-Laplacian and $(2,p')$-Laplacian eigenvalue problems: 
\[ ([u']^{p-1})'=\lambda u,\ 0\le t \le b \qquad \mbox{ with } u(0)=u'(b)=0  \]
and 
\[ u''=\lambda [u]^{p'-1},\ 0\le t \le b \qquad \mbox{ with } u'(0)=u(b)=0,  \]
respectively (here $[u]^p:= |u|^{p-2}u$ correspond to the duality mapping $J_{L_p}$.)   

As a consequence of  the definition and properties of j-eigenfunctions for map $H$, we find that all non-zero multiples of $\sin_{p,2}({\pi_{p,2}x \over 2b})$, which are just the image of $\cos_{p,2}({\pi_{p,2}x \over 2b})$ under the map $H$ multiplied by an appropriate constant  (see \cite{EGL} for more about the relation between $\sin_{p,q}$ and $\cos_{p,q}$ functions) are extremals of $H^*$ and then equal to a constant multiple of $\cos_{2,p'}({\pi_{2,p'}(b-x) \over 2b})$. This is in agreement with  \cite[Prop. 3.1, 3.2]{EGL}.

Now, by considering the Hilbertian map $K:=H^*H:L_p(0,b) \to L_p(0,b)$, we can see that the extremal function for this map is equal to the first eigenfunction for the $(p,p')$-bi-Laplacian operator:
\[ ([u'']^{p-1})'' = \lambda [u]^{p'-1}, \qquad 0 \le t \le b
\]
\[u(0)=u'(T)=[u''(0)]^{p-1}=[[u''(b)]^{p-1}]'=0.
\]
By using information about the extremal functions for maps $H$ and $H^*$ and their relation it is quite simple to see that the extremal functions and the above eigenvalue must be a non-zero multiple of $\sin_{2,p'}({\pi_{2,p'}x \over 2b}).$ We can check this by a direct computation with the help of \cite[Lemma 2.2, Prop. 3.1, 3.2]{EGL}. 

To the best of our knowledge this is, outside the trivial case for the (2,2)-bi-Laplacian operator, the only known case in which  the first eigenfunction of (p,q)-bi-Laplacian operator is exactly described. This example raises the obvious questions whether eigenfunctions for general (p,q)-bi-Laplacian operators can be expressed directly and how different is their structure from that of eigenfunctions for  (p,q)-Laplacians. These matters will be addressed in a forthcoming paper by L.Boulton and J.Lang.

Now we recall two operator ideals which can be used to split the set of compact operators. 
The members of both these operator ideals are called  $p-$compact operators, which can bring some confusion. Pietsch was the first to use an ideal defined using this
terminology, and we shall call its members   $\mathfrak{p \text{-} compact}$ maps. A more recent definition of $p-$compactness was introduced by by Sinha and Karn and we will simply say that the corresponding maps are $p-$compact. We start by recalling the more recent definition.

Let $p\in(1,\infty)$ and suppose that $X,Y$ are Banach spaces. Following
\cite[Definition 2.1]{SK} we say that a subset $K$ of a Banach space $X$ is \textit{relatively
}$p-$\textit{compact} if there is a sequence\textit{ }$\left\{  x_{n}\right\}
$ in $X,$ with $\left\{  \left\Vert x_{n}\right\Vert _{X}\right\}  \in l_{p},$
such that
\[
K\subset\left\{  \sum\nolimits_{n=1}^{\infty}\alpha_{n}x_{n}:\sum
\nolimits_{n=1}^{\infty}\left\vert \alpha_{n}\right\vert ^{p^{\prime}}%
\leq1\right\}  ,
\]
where $1/p+1/p^{\prime}=1.$ When $p=1$ this definition takes the form%
\[
K\subset\left\{  \sum\nolimits_{n=1}^{\infty}\alpha_{n}x_{n}:\sup
_{n}\left\vert \alpha_{n}\right\vert \leq1\right\}  ,
\]
where $\left\{  \left\Vert x_{n}\right\Vert _{X}\right\}  \in l_{1}.$ These
definitions stem from a classical result of Grothendieck \cite{Gro} which
states that $K$ is relatively compact if and only if if there is a
sequence\textit{ }$\left\{  x_{n}\right\}  $ in $X,$ with $\left\{  \left\Vert
x_{n}\right\Vert _{X}\right\}  \in c_{0},$ such that%
\[
K\subset\left\{  \sum\nolimits_{n=1}^{\infty}\alpha_{n}x_{n}:\sum
\nolimits_{n=1}^{\infty}\left\vert \alpha_{n}\right\vert \leq1\right\}  .
\]
With this in mind we can say that relatively compact sets are relatively
$\infty$-compact. It is easy to see that  if $1\leq p\leq q\leq\infty,$ then a relatively
$p$-compact set is relatively $q$-compact (see \cite{SK}).  Corresponding to this notion are
the $p$-compact operators, defined in \cite[Definition 2.2]{SK} as follows:\newline

\begin{definition}
	\label{Definition 3.1}Given $p\in\lbrack1,\infty],$ and Banach spaces $X, Y,$ a map $T\in B(X,Y)$ is
	said to be $p$-compact if $T(B)$ is relatively $p$-compact whenever $B$ is a
	bounded subset of $X.$
\end{definition}

An equivalent form of this condition is that there is a sequence\textit{
}$\left\{  x_{n}\right\}  $ in $Y,$ with $\left\{  \left\Vert x_{n}\right\Vert
_{Y}\right\}  \in l_{p},$ such that
\[
T(B_{X})\subset\left\{  \sum\nolimits_{n=1}^{\infty}\alpha_{n}x_{n}%
:\sum\nolimits_{n=1}^{\infty}\left\vert \alpha_{n}\right\vert ^{p^{\prime}%
}\leq1\right\}
\]
when $p\in(1,\infty),$ with obvious adaptations when $p=1$ or $\infty.$ A norm
is introduced on the family $K_{p}=K_{p}(X,Y)$ of all $p$-compact
\ maps from $X$ to $Y$ by
\[
\left\Vert T\right\Vert _{K_{p}}:=\inf\left\Vert \left\{  \left\Vert
x_{n}\right\Vert_Y \right\}  \right\Vert _{l_{p}},
\]
where the infimum is taken over all sequences $\left\{  x_{n}\right\}  $ in
the above definition. With this norm $K_{p}$ is a Banach operator ideal (see
\cite{Oja}). Note that $K_{\infty}$ coincides with the ideal of compact operators and $K_1$ coincides with of nuclear operators (which we will denote by  $\mathcal{N}$).   
Also we have $K_p \subset K_q$ whenever $p < q$.
From \cite[Theorem 1]{Pie3} we have the following connection between $K_p$ operator ideals and $(r,s,q)$-nuclear operators:

\[K_p=\mbox{ surjective hull of } \mathcal{N}(p,1,p). \]

At this point we  recall the term \textquotedblleft%
$p$-compact map\textquotedblright as it was used many years ago by Pietsch in
\cite{Pie1}, section 18.3.1. We shall refer to the condition required by him as p-compactness in the sense of Pietsch or shortly $\mathfrak{p \text{-} compactness}$. By Theorem 18.3.2 of \cite{Pie1}, $T\in B(X,Y)$ is
$p$-compact in the sense of Pietsch if and only if there are maps $A\in K(X,l_{p})$ and
$B\in K(l_{p},Y)$ such that $T=B\circ A.$ The family of all such maps will be
denoted by $\ \mathfrak{K}_p=\mathfrak{K}_{p}(X,Y)$ and
\[
\left\Vert T\right\Vert _{\mathfrak{K}_{p}}:=\inf\left\Vert B\right\Vert
\left\Vert A\right\Vert ,
\]
where the infimum is taken over all possible factorisations $T=B\circ A$ of
$T.$ With this norm $\mathfrak{K}_{p}$ is a Banach operator ideal; and it
follows from \cite{Pie1}, 18.3.2 and 18.1.3 \ that $\left(  \mathfrak{K}_{{p}%
},\left\Vert . \right\Vert _{\mathfrak{K}_{p}}\right)  $ coincides with
the operator ideal $\mathcal{N}\left(  \infty,p,p^{\prime}\right)  $ of
$\left(  \infty,p,p^{\prime}\right)  -$nuclear operators. For  each $T\in
\mathfrak{K}_{{p}}$ there is a representation $T=B\circ S_{0}\circ A$, where
$A\in B(X,l_{p}),A\in B(l_{p},Y)$ and $S_{0}\in B(l_{p})$ is a diagonal
operator of the form $S_{0}\left(  \left\{  \xi_{n}\right\}  \right)
=\left\{  \sigma_{n}\xi_{n}\right\}  $ with $\left\{  \sigma_{n}\right\}  \in
c_{0}.$ For discussions of the relations and differences between the ideal of
$p$-compact maps and that of $\mathfrak{p \text{-} compact}$ maps see \cite{Oja} and
the last paragraph of \cite{Pie3}.
Obviously we have  that every $\mathfrak{2 \text{-} compact}$ map is a compact Hilbertian map { (the converse is not true)}. 

\begin{proposition}
	\label{Proposition 3.6}\ The following inclusions hold:%
	\[
	\mathcal{N}\subset K_{2}\subset {\mathfrak{K_2}}\subset\Gamma_{2}.
	\]
	
\end{proposition}

\begin{proof}
	We choose to give a proof that uses the techniques of operator ideals. In the
	notation of \cite{Pie1}, 18.1 in which $\mathcal{N}(r,p,q)$ stands for the
	class of all $(r,p,q)-$nuclear maps from $X$ to $Y,$ we have $\mathcal{N}%
	=\mathcal{N}(1,1,1);$ and by Proposition 18.1.5 of \cite{Pie1},
	\[
	\mathcal{N}(1,1,1)\subset\mathcal{N}(2,1,2)\subset\mathcal{N}(\infty,2,2).
	\]
	From \cite{Oja}, p. 952, $K_{2}=$ $\mathcal{N}(2,1,2);$ from our results
	above, $\mathfrak{K_{{2}}}=\mathcal{N}(\infty,2,2).$ The proof is complete.
\end{proof}

As an immediate consequence we have the following:
\begin{corollary}
Let $X,Y$ be uniformly smooth and uniformly convex Banach spaces and suppose that $T \in {\mathfrak{K_2}}(X,Y)$ or $T\in K_p(X,Y)$ where $1\le p \le 2$. Then, with the help of $j-$eigenfunctions, we have
	\begin{equation} \label{The main result 02}
	T(x)= \sum_{i=1}^{\infty} \lambda_i^{\perp} y_i^{\perp} \langle x, x_i'^{\perp} \rangle_X,  \qquad \left(  x\in X\right), 
\end{equation}
\end{corollary}

Since we were not able to find  many examples of $p-$compact operators in the literature we close this section with some concrete illustrations.

{\bf Examples} \\
(1) We consider the classical Hardy operator.

Thus let $I=(0,1)$ and $p\in (1,\infty );$ define $T:L_{2}(I)\rightarrow
L_{p}(I)$ by  
\[
(Tf)(s)=\int_{0}^{s}f(t)dt,s\in I.
\]%
Define functions $f_{n}\in L_{2}(I)$ by 
\[
f_{n}(t)=\left\{ 
\begin{array}{cc}
	1, & n=1, \\ 
	\cos (n\pi t), & n>1.%
\end{array}%
\right. 
\]%
Then $\left\{ f_{n}\right\} $ is an orthogonal basis of $L_{2}(I);$ given
any $f\in B$ (the unit ball in $L_{2}(I)$), 
\[
f=\sum_{n=1}^{\infty }c_{n}f_{n},
\]%
where 
\[
\left\Vert f\right\Vert _{2}^{2}=\sum_{n=1}^{\infty }\left\vert
c_{n}\right\vert ^{2}\left\Vert f_{n}\right\Vert _{2}^{2}=\left\vert
c_{1}\right\vert ^{2}+\frac{1}{2}\sum_{n=2}^{\infty }\left\vert
c_{n}\right\vert ^{2}
\]%
since%
\[
\left\Vert f_{n}\right\Vert _{2}^{2}=\int_{0}^{1}\cos ^{2}(n\pi t)dt=\frac{1}{2}\mbox{ }\left( n>1\right) .
\]%
Let $f\in B.$ Then $Tf=\sum_{n=1}^{\infty }c_{n}Tf_{n}$ and 
\[
(Tf_{n})(s)=\int_{0}^{s}f_{n}(t)dt=\left\{ 
\begin{array}{cc}
	s, & n=1, \\ 
	\frac{1}{n\pi }\sin (n\pi s), & n>1.%
\end{array}%
\right. 
\]

Thus if $n>1,$ 
\begin{eqnarray*}
	\left\Vert Tf_{n}\right\Vert _{p}^{p} &=&\int_{0}^{1}\left\vert \frac{\sin
		(n\pi s)}{n\pi }\right\vert ^{p}ds=(n\pi )^{-1-p}\int_{0}^{n\pi }\left\vert
	\sin t\right\vert ^{p}dt \\
	&=&n(n\pi )^{-1-p}\int_{0}^{\pi }\left\vert \sin t\right\vert ^{p}dt \\
	&=&n(n\pi )^{-1-p}\left\{ \int_{0}^{\pi /2}\sin ^{p}tdt+\int_{\pi /2}^{\pi
	}\sin ^{p}tdt\right\}  \\
	&=&n(n\pi )^{-1-p}\left\{ \int_{0}^{\pi /2}\sin ^{p}tdt+\int_{0}^{\pi
		/2}\cos ^{p}tdt\right\}  \\
	&=&Cn^{-p}.
\end{eqnarray*}

It follows that \ if $q\in (1,2],$ $Tf=\sum_{n=1}^{\infty }c_{n}Tf_{n},$
where $\left\{ c_{n}\right\} \in l_{2}\hookrightarrow l_{q^{\prime }}$ and $%
\left\{ Tf_{n}\right\} \in l_{q}\left( L_{p}(I)\right)$: in other words, $T$
is $q-$compact for all $q\in (1,2]$ and hence for all $q\in (1,\infty ].$
 By Theorem 4 of {\cite{EGL}}, $T:L_{p}(I)\rightarrow L_{p}(I)$
is not nuclear $\left( 1<p<\infty \right) $ and so is not $1-$compact. It
follows that $T:L_{2}(I)\rightarrow L_{2}(I)$ is $q-$compact for all $q\in
(1,\infty ]$ \ but not nuclear.

When  $T$ $\ $is regarded as a map from $L_{p}(I)\rightarrow L_{q}(I),$ with 
$p,q\in (1,\infty ),$ we use the fact that the generalised trigonometric
functions $g_{n}$ defined by $g_{n}(t)=\cos _{p,p^{\prime }}(n\pi
_{p,p^{\prime }}t)$ $\left( n\in \mathbb{N}_{0}\right) $ form a basis of $%
L_{p}(I)$ if $p\in \lbrack p_{1},p_{2}],$ where $p_{1}\in (1,2)$ and $%
p_{2}\in (2,\infty )$ \ are calculable numbers (see \cite{EM} and \cite{EL2}). \ We suppose
henceforth that $p$ is constrained in this way. Then all $g$ in the unit
ball $B_{p}$ of $L_{p}(I)$ can be represented in the form 
\[
g=\sum_{n=1}^{\infty }c_{n}g_{n}.
\]%
Calculations similar to those for $Tf_{n}$ in the last case (see also \cite%
{EL}, p.40) show that the basis $\left\{ g_{n}\right\} $ is seminormalised
in the sense that $0<\inf \left\Vert g_{n}\right\Vert _{p}\leq \sup
\left\Vert g_{n}\right\Vert _{p}<\infty .$ Thus by Proposition 1.2.18 of \cite%
{EE}, $\left\{ c_{n}\right\} \in l_{r^{\prime }}$ for some $r^{\prime }\in
(1,\infty ).$ As in the previous case it can be shown that $\left\Vert
Tg_{n}\right\Vert _{q}\leq Kn^{-1},$ so that $\left\{ Tg_{n}\right\} \in
l_{r}\left( L_{q}(I)\right) :\ \ $thus $T$ is $r-$compact. We conclude that
if $p$ is not too far from $2,$ more precisely $p\in \lbrack p_{1},p_{2}],$
then there exists $r$ $\in (1,\infty )$ such that $T$ is $s-$compact for all 
$s\geq r.$
\
\\
\
\\
(2) Another example is provided by a simple Sobolev embedding. Let $Q=(-\pi ,\pi
)$ and $B=(-1,1);$ let $W_{2}^{1}(Q)$ be the usual (complex) first-order
Sobolev space \ based on $L_{2}(Q),$ and put%
\[
X = \{ f\in W_{2}^{1}(Q): \mbox{supp } f\subset \overline{B} \} .
\]
Endowed with the inherited norm, $X$ is a closed subspace of $W_{2}^{1}(Q).$
Denote by $I$ the natural embedding of $X$ in $L_{2}(Q)$ and let $f$ belong
to the unit ball of $X.$ Then (see \cite{HT}, 4.4) 
\[
(If)(x)=\sum_{m\in \mathbb{Z}}a_{m}h_{m}(x)\mbox{ \ }\left( x\in Q\right) ,
\]%
where $h_{m}(x)=\left( 2\pi \right) ^{-1/2}e^{imx};$ $\left\{ h_{m}\right\} $
is an orthonormal basis of $L_{2}(Q);$ moreover, 
\[
\sum_{m\in \mathbb{Z}}\left( 1+m^{2}\right) \left\vert a_{m}\right\vert
^{2}\sim \left\Vert f\right\Vert _{X}^{2}=1.
\]%
Thus 
\[
If=\sum_{m\in \mathbb{Z}}\left( 1+m^{2}\right) ^{1/2}a_{m}\cdot \left(
1+m^{2}\right) ^{-1/2}h_{m},
\]%
where $\left\{ \left( 1+m^{2}\right) ^{1/2}a_{m}\right\} \in l_{2}$ and $%
\left\{ \left\Vert \left( 1+m^{2}\right) ^{-1/2}h_{m}\right\Vert
_{2,Q}\right\} \in l_{2}.$ Hence $I$ is $2-$compact and so is $q-$compact for
all $q\geq 2.$

\bigskip

\bigskip D. E. Edmunds, Department of Mathematics, University of Sussex,
Brighton BN1 9QH, U.K.

J. Lang, Department of Mathematical Sciences, Ohio State University, Columbus,
OH 43210-1174, U.S.A.

\ 
\end{document}